\documentclass[11pt,a4paper]{article}
\usepackage{color}
\usepackage{graphicx}
\usepackage{graphics}
\usepackage{makeidx}
\usepackage{showidx}
\usepackage{latexsym}
\usepackage{amssymb}
\usepackage{verbatim}
\usepackage{amsmath}
\usepackage{amsthm}
\usepackage{amsfonts}
\usepackage{amssymb,amsmath}
\usepackage{latexsym,amsthm,amscd}
\usepackage[all]{xy}
\newtheorem{thm}{Theorem}
\newtheorem{prop}[thm]{Proposition}
\newtheorem{defn}{Definition}
\newcounter{alphthm}
\newtheorem{propriete}[alphthm]{Theorem}

\newtheorem{cor}{Corollary}

\newtheorem{lem}[thm]{Lemma}

\newcommand{\be}{\begin{equation}}
\newcommand{\ee}{\end{equation}}
\newcommand{\ben}{\begin{enumerate}}
\newcommand{\een}{\end{enumerate}}
\newcommand{\beq}{\begin{eqnarray}}
\newcommand{\eeq}{\end{eqnarray}}
\newcommand{\beqn}{\begin{eqnarray*}}
\newcommand{\eeqn}{\end{eqnarray*}}

\newcommand{\bpf}{\begin{proof}}
\newcommand{\epf}{\end{proof}}
\newcommand{\bl}{\begin{lem}}
\newcommand{\el}{\end{lem}}
\newcommand{\bp}{\begin{prop}}
\newcommand{\ep}{\end{prop}}
\newcommand{\bd}{\begin{defn}}
\newcommand{\ed}{\end{defn}}
\newcommand{\bt}{\begin{thm}}
\newcommand{\et}{\end{thm}}

\newcommand\bpr{\begin{prop}}
\newcommand\epr{\end{prop}}

\title{\bf \small A COMPARISON THEOREM ON PROJECTIVE FINSLER GEOMETRY}
\author{M. Sepasi and B. Bidabad\footnote{corresponding author}}
\date{\small Faculty of Mathematics and Computer Science, Amirkabir University of Technology, (Tehran Polytechnic), Hafez Ave., 15914 Tehran, Iran.\\ m\_sepasi@aut.ac.ir;\\ bidabad@aut.ac.ir}
\usepackage[pdftex,bookmarks,colorlinks]{hyperref}
\begin{document}
 \maketitle
\begin{abstract}
Here, a non-linear analysis method is applied rather than classical one to study projective Finsler geometry.  More intuitively, by means of an inequality on Ricci-Finsler curvature, a projectively invariant pseudo-distance is introduced and an analogous of  Schwarz' lemma in Finsler geometry is proved. Next, the Schwarz' lemma is applied to show that the introduced pseudo-distance is a distance. This projectively invariant distance  will  be served in continuation of this work to investigate  Einstein-Finsler spaces and classify Finsler spaces as well.
\end{abstract}
{\small \bf keywords: }{Ricci tensor, Einstein-Finsler space, projective parameter, Funk distance,  Schwarzian derivative.}

{\small Mathematic Subject Classification: 53B40,  58B20.}

\section{Introduction}
In mathematics, a geodesic as a generalization of straight line determines geometry of the space and in physics, a geodesic represents the
equation of motion, which describes all the phenomena.
If two regular affine connection  on a manifold  have the same geodesics as the point sets, then they are said to be projectively related.
Much of the practical importance of two projectively related ambient spaces derives from the fact that they produce same physical events, see for instance \cite{AM}.


In projective geometry, there are two different well known approaches. The classical method is application of projectively invariant quantities, cf., \cite{AIM,3,4}.
Another approach is application of projectively invariant  distance functions
. For instance in Riemannian geometry see  \cite{1,8}.

The present work is motivated by two distinct aims. First, an endeavor has been made to furnish a reasonably comprehensive account  of analysis based on the methods of Schwarzian derivative and projective invariant distance, in sense of the second approach, on Finsler geometry.
Next,  this monograph will be  served as an introduction to our following work in classification of  Einstein-Finsler spaces.

One of the present author in a recent joint work introduced a conformally invariant distance function, determined by electrostatic capacity of a condenser, to present a classification of Finsler spaces, cf. \cite{BH}.
Here in this paper a projectively invariant distance is  defined in Finsler spaces and some basic results are obtained. In fact, inspired by Berwald's method, we first introduce a projectively invariant parameter for  geodesics on a Finsler space.
More intuitively, let $\gamma(t)$ be a geodesic of an affine connection on a manifold. In general, the parameter $t$ does not remain invariant under projective changes. There is a unique parameter up to linear fractional  transformations which is projectively invariant. This parameter is referred to, in the literature, as \emph{projective parameter}. See \cite{5,6,7} for a survey. A mapping $f:I\rightarrow M$ is said to be \emph{projective map} if $f$ describes a geodesic on $M$ and its natural parameter is a projective parameter.
In Ref. \cite{8}, Kobayashi used the projective map to introduce a projectively invariant pseudo-distance $d_M$, on a connected Riemannian space and proved the following version of Schwarz's lemma.
\begin{propriete}
Let $(M,g)$ be a Riemannian space for which the  Ricci tensor satisfies
$$(R_{jk})\leq- c^2(g_{jk}),$$
as matrices, for a positive constant $c$. Then every projective map $f:I=]-1,+1[\rightarrow M$ satisfies
$$f^{*}ds_M^2 \leq \frac {n-1}{4c^2} d{s_I}^2,$$
where,  $ds_M^2 = \Sigma g_{jk}dx^jdx^k$  and $d{s_I}^2$ are the  first fundamental forms of $g$ and Poincar\'{e} metric on $I$.
\end{propriete}
Next he used the Schwarz' lemma  to prove  the following theorem.
\begin{propriete}
 Let $(M,g)$ be a (complete) Riemannian space for which the  Ricci tensor satisfies
$$(R_{jk})\leq- c^2(g_{jk}),$$
as matrices, for a positive constant $c$. Then the pseudo-distance $d_M$, is a (complete) distance.
\end{propriete}
This result is used later by Kobayashi to characterize Einstein-Riemann spaces as follows;
\emph{ The projective transformations of a complete Einstein space with negative Ricci tensor are all isometries.}, cf., \cite{1}.

In this paper, with above objective in mind, and by means of Funk metric, a pseudo-distance on  connected Finsler spaces is introduced. And a  Finslerian setting of the Schwarz' lemma is carried out  as follows;
 \setcounter{thm}{0}
\bt 
  Let $(M,F)$ be a connected Finsler space for which the Ricci tensor satisfies
  \be \nonumber
 (Ric)_{ij} \leq -c^2g_{ij},
  \ee
 as matrices, for a  positive constant $c$.
  Then we have
  \begin{equation*}
  {\tilde{f}}^*(ds_M^2) \leq \frac{(n-1)k^2}{4c^2} ds_I^2,
  \end{equation*}

     where, $ds_I$ and $ds_M$ are the first fundamental  forms of the Funk metric on $I$ and  the Finsler metric on $M$ respectively, and  $\tilde{f}$ is the natural lift of an arbitrary projective map ${f}$.
 \et
 Next, the  Showarz' lemma is used to prove the following theorem.

  \bt
  Let $(M,F)$ be a connected Finsler space for which the Ricci tensor satisfies
  \begin{equation*}
  (Ric)_{ij} \leq -c^2g_{ij},
  \end{equation*}
 as matrices, for a positive constant $c$. Then the pseudo-distance $d_M$, is a distance.
  \et

\section{Preliminaries }
Let $M$ be an $n$- dimensional  $C^{\infty}$ manifold, $(x , U)$ a local chart on $M$, and $TM$ the bundle of tangent spaces $TM:={\cup}_{x\in M} T_xM$.  Elements of $TM$ are denoted by  $(x,y)$ and called line element, where $x\in M$ and $y\in T_xM$. The natural projection $\pi:TM\rightarrow M $, is given by $\pi (x,y):= x$.

 The pull-back tangent bundle $\pi^* TM$ is a vector bundle over the slit tangent bundle $TM_0:=TM\backslash 0$ whose fiber $\pi^*_v TM$ at $v \in TM_0$ is just $T_xM$, where $\pi (v) = x$. Then
${\pi}^{*} TM = \{(x,y,v) \mid y\in T_xM, v \in T_xM\}.$

A (globally defined) Finsler structure on $M$ is a function $F: TM\rightarrow [0 , \infty) $ with the following properties;
\begin{itemize}
\item[(i)] Regularity: $F$ is $C^{\infty}$ on the entire slit tangent bundle $TM_0$,

\item[(ii)] Positive homogeneity:  $F(x , \lambda y) = \lambda F(x , y)$ for all $\lambda > 0$,

\item[(iii)]Strong convexity: The Hessian matrix $(g_{ij}) := ({[1/2F^2]}_{y^iy^j})$, is positive-definite at every point of $TM_0$.
\end{itemize}
The pair $(M,F)$ is known as a Finsler space.

Let $\gamma : [a , b] \rightarrow M$ be a piecewise $C^{\infty}$ curve with velocity $\frac{d\gamma}{dt} = \frac{d{\gamma}^i}{dt} \frac{\partial}{\partial x^i} \in T_{\gamma (t)} M$. The arc length parameter of $\gamma$ is given by
\be\label{arc length}
s(t) = \int_{t_0}^{t} F(\gamma , \frac{d\gamma}{dr}) dr.
\ee
 Its integral length  is denoted by $L(\gamma) := \int_a^b F(\gamma , \frac{d\gamma}{dt}) dt $.
 For $x_0$ , $x_1$ $\in M$, denote by $\Gamma (x_0 , x_1)$ the collection of all piecewise $C^{\infty}$ curves  $\gamma : [a , b] \rightarrow M$ with $\gamma (a) = x_0$ and $\gamma(b) = x_1$. Define a map $d_F : M \times M \rightarrow [0 , \infty)$ by
\be\label{distance}
d_F(x_0 , x_1) := inf L(\alpha),  \quad \alpha \in \Gamma (x_0 , x_1).
\ee
It can be shown that $d_F$ satisfies the first two axioms of a metric space. Namely,
\begin{itemize}
\item[(1)] $d_F(x_0 , x_1) \geq 0$ , where equality holds if and only if $x_0 = x_1$,
\item[(2)] $d_F(x_0 , x_1) \leq d_F(x_0 , x_1) + d_F(x_1 , x_2)$.
\end{itemize}
We should remark that the distance function $d_F$ on a Finsler space does not have  the symmetry property.
If the Finsler structure F is absolutely homogeneous, that is $F(x,\lambda y)=\mid \lambda \mid F(x,y)$ for $\lambda \in \mathbb{R}$, then one also has
\begin{itemize}
\item[(3)] $d_F(x_0, x_1) = d_F(x_1,x_0)$.
\end{itemize}
See Ref. \cite{9}.

Consider $\gamma_{jk}^i:=1/2g^{is}(\frac{\partial g_{sj}}{\partial x^k}-\frac{\partial g_{jk}}{\partial x^s}+\frac{\partial g_{ks}}{\partial x^j})$, the formal Christoffel symbols of the second kind, and let $G^i:={\gamma}^i_{jk}y^iy^j$.
 A $C^\infty$ curve $\gamma: t\rightarrow x^i(t)\in M$ is called a \emph{geodesic} of the  Finsler space $(M,F)$,
if it obeys the system of differential equation
\begin{equation}\label{e3}
\frac{d^2x^i}{dt^2} + G^i(x(t), \frac{dx}{dt}) = f(t) \frac{dx^i}{dt},
\end{equation}
where, $f(t)=\frac{d^2s}{dt^2}/ \frac{ds}{dt}=\frac{d}{dt}[\log F(T)]$, $s$ is the arc length parameter, and $T$ is the velocity field, cf. Ref. \cite{9}.
Replacing the arbitrary parameter $t$ by  the arc length parameter $s$ the above equation reads
\begin{eqnarray}\label{e2}
\frac{d^2x^i}{ds^2}+G^i(x(s), \frac{dx}{ds}) = 0.
 \end{eqnarray}

Let $\bar{F}$ be another Finsler structure on M. If any geodesic of $(M,F)$  coincides with a geodesic of
$ (M,\bar{F})$ as set of points and vice versa, then the change  $F\rightarrow \bar{F}$ of the metric is called \emph{projective } and $F$ is said to be \emph{projective} to $\bar{F}$.
      A Finsler space $(M,F)$  is projective to another Finsler space $(M,\bar{F})$, if and only if there exists a  1-homogeneous scalar field $P(x,y)$ satisfying
   \begin{equation}\label{e4}
    \bar{G}^i(x , y)=G^i(x , y)+P(x,y)y^i.
    \end{equation}

    The scalar field $P(x , y)$ is called the \emph{projective factor} of the projective change under consideration. Let  $G_j^i := \frac{\partial G^i}{\partial y^j}$  and ${\ell}^j := \frac{y^j}{F}$. If we put
 \begin{equation}\label{e5}
 {R^i}_k:= \frac{1}{2}{\ell}^j (\frac{\delta}{\delta x^k}\frac{G^i_j}{F}-\frac{\delta}{\delta x^j} \frac{G^i_k}{F}),
 \end{equation}
then it can be easily shown
 \begin{equation}\label{e6}
2F^2 {R^i}_k=2 (G^i)_{x^k}-\frac{1}{2}(G^i)_{y^j}(G^j)_{y^k}-y^j(G^i)_{y^kx^j}+G^j(G^i)_{y^ky^j}.
  \end{equation}
See Ref. \cite[ P.71]{9}.
The \emph{Ricci Scalar} is defined by $Ric:= {R^i}_i$.

 Here we use the definition of  \emph{Ricci tensor}
introduced  by Akbar-Zadeh,   as follows
 \begin{equation}\label{e8}
Ric_{ik} :=\frac{1}{2}(F^2Ric)_{y^iy^k}.
\end{equation}
Moreover, by homogeneity we have $Ric_{ik} {\ell}^i{\ell}^k = Ric$.
From (\ref{e6})  we obtain
\be\label{e9}
2F^2 Ric = 2 (G^i)_{x^i} -\frac{1}{2} (G^i)_{y^j}(G^j)_{y^i}- y^j(G^i)_{y^ix^j} + G^j(G^i)_{y^iy^j}.
\end{equation}
Under the projective change (\ref{e4}) we have
\begin{equation}\label{e10}
\bar{F}^2\bar{Ric} = F^2Ric + \frac{(n-1)}{2}(\frac{\partial P}{\partial x^i} y^i - \frac{\partial P}{\partial y^i} G^i + \frac{P^2}{2}).
\end{equation}
 Now we are in a position to define in the next section  the projective parameter of a geodesic on a Finsler space.
\section{A projectively invariant distance on a Finsler space}
\subsection{Projective Parameter}
 Berwald in Ref. \cite{5}, introduced the notion of a \emph{general affine connection $\Gamma$} on an $n$-dimensional manifold $M$, as a geometric object  with components ${\Gamma}^i_{jk}(x , \dot{x})$, 1-homogeneous in $\dot{x}$.  These geometric objects transform by the local change of  coordinates
\be\label{coordinatesChange}
{\bar{x}}^i = {\tilde{x}}^i(x^1,...,x^n),
\ee
as ${\tilde{\Gamma}}^i_{jk} = ({\Gamma}^l_{mr} \frac{\partial x^m}{\partial{\tilde{x}}^j} \frac{\partial x^r}{{\tilde{x}}^k} + \frac{{\partial}^2 x^l}{\partial{\tilde{x}}^j \partial{\tilde{x}}^k}) \frac {\partial {\tilde{x}}^i}{\partial x^l}$,
wherever $\dot{x}$ are transformed like the components of a contravariant vector.

Berwald has defined the projective parameter  for geodesics of general affine connections to be a parameter which is  projectively invariant.
These specifications are carefully spelled out for geodesics of Finsler metrics in the following natural manner.


First recall that for a $C^\infty$ real function $f$ on $\mathbb{R}$, and for $t \in \mathbb{R}$, the \emph{Schwarzian derivative}
$$\{f , t\} = \frac{\frac{d^3f}{dt^3}}{\frac{df}{dt}}-\frac{3}{2}\big[\frac{\frac{d^2f}{dt^2}}{\frac{df}{dt}}\big]^2, $$
 is defined to be an operator which is invariant under all linear fractional transformations
 $t \rightarrow \frac{a t + b}{c t +d}$ where, $ad - bc \neq 0$. That is,
 \be\label{PropertyOfSchwarzian}
 \{ \frac{a f + b}{c f +d} , t \} = \{f , t\}.
 \ee
Let $g$ be a real function for which the composition $f\circ g$ is defined. Then,
\be\label{schwarzian}
\{f\circ g,t\}=\{f,g(t)\}(\frac{dg}{dt})^2+\{g,t\}.
\ee
In general, the parameter $t$ of a geodesic, does not remain invariant under projective changes. Here, we show that there is a unique parameter up to linear fractional  transformations which is projectively invariant. This parameter is referred to, in the literature, as projective parameter.

Let $\gamma$  be a geodesic on the Finsler space $(M , F)$. We need  a parameter $\pi$ which remains invariant  under both the coordinates change (\ref{coordinatesChange}),  and   the projective change (\ref{e4}).
We define the \emph{projective normal parameter} $\pi$ of   $\gamma$  by
\begin{equation}\label{e11}
\{\pi , s\} = -4 A G^0(x , \frac{dx}{ds}),
\end{equation}
where $\{\pi , s\}$ is the Schwarzian derivative,
 $A\neq0$ is a constant and
$ G^0(x , \dot{x})$  is a  homogeneous function of second degree in $\dot{x}$.
We require that the parameter $\pi$ remains invariant under the coordinates change (\ref{coordinatesChange}),  and   the projective change (\ref{e4}). This gives to the quantity $G^0$ the following   transformation laws,
 \begin{equation}\label{e12}
 \tilde{G}^0(\tilde{x} , \dot{\tilde{x}}) = G^0(x , \dot{x})\qquad \dot{\tilde{x}}^i= \frac{\partial \tilde{x}^i}{\partial x^k} \dot{x}^k.
 \end{equation}
By projective change (\ref{e4}), we have
 \begin{equation}\label{e13}
\bar{G}^0 = G^0 - \frac{1}{4A}(\frac{\partial P}{\partial x^i} \dot{x}^i - \frac{\partial P}{\partial \dot{x}^i} G^i + \frac{P^2}{2}).
\end{equation}

According to  (\ref{e10}) and (\ref{e13}),  the scalar $R^*$ defined by
 \begin{equation}\label{e14}
R^* := F^2Ric + 2A (n-1)G^0,
\end{equation}
 is 2-homogeneous in $\dot{x}^i$ and remains invariant by projective change (\ref{e4}).
If we put $R^* = 0$ then
 \begin{equation}\label{e15}
G^0 = -\frac{1}{2A(n-1)}F^2 Ric.
\end{equation}
Plugging the value of $G^0$ into  (\ref{e11}), we obtain
\begin{equation}\label{e16}
\{\pi,s\}=\frac{2}{n-1}F^2Ric = \frac{2}{n-1} Ric_{jk}\frac{d{x}^j}{dt}\frac{d{x}^k}{dt},
\end{equation}
which is called the  \emph{preferred projective normal parameter} up to linear fractional transformations.
In the sequel we will simply refer to preferred projective normal parameter as \emph{ projective parameter}.

 Let $(M,F)$ be projectively related to $(M,\bar F)$ and  the curve $\bar{x}(\bar{s})$ be a geodesic with affine parameter  $\bar{s}$ on $(M,\bar F)$  representing the same geodesic as $x(s)$  of $(M,F)$, except for its parametrization. Then, one can easily check that a projective parameter $\bar{\pi}$ defined by $\bar{x}(\bar{s})$ is related to the projective parameter $\pi$  by  $\bar{\pi} = \frac{a\pi+b}{c\pi+d}$.

\subsection{Funk distance and Funk metric}
  Let $D$ be a convex domain in $\mathbb{R}^n$ and $\partial D$  its boundary.
For any two points $A$ and $B$ on $D$, the line through $A$ and $B$ intersects $\partial D$ at $P$ in the order $A$, $B$ and $P$. The \emph{Funk's distance} $f(A , B)$ is defined by  $f(A , B) :=\frac {1}{k} \log \frac{AP}{BP}$, where $k$ is a positive constant and $AP$ and $BP$ denote the Euclidean distances, cf. Ref. \cite{11}.
Clearly, the distance $f$ satisfies
\begin{itemize}
\item $f(A , B)\geq 0$ for any two points $A$ and $B$ in $D$.
\item $f(A , B) = 0$ if and only if $A=B$.
\item $f(A , B) + f(B , C) \geq f(A , C)$ for any three points $A$, $B$ and $C$ in $D$. The equality holds if and only if $B$ is on the segment $AC$, provided  $D$ is strictly convex.
\item $f(A , B) \neq f(B , A)$ in general, but $f(A , A_n)\rightarrow 0$ if and only if $  f(A_n , A)\rightarrow 0.$
\end{itemize}
\setcounter{thm}{0}
\bp
 Let $D = \{(x^i) \in \mathbb{R}^n \mid \phi (x^i) > 0\}$, and $ \partial D : \phi(x^i) = 0$,
where $ \phi(x^i) = \alpha_{ij} x^ix^j + 2\beta_i x^i + \gamma$,  $\alpha_{ij} = \alpha_{ji}$ and $\gamma > 0$ is a positive number.
Then the Funk metric $L_f$, is given by
\begin{equation}\label{e17}
L_f(x,y) = \{(a_{ij}(x) y^iy^j)^{1/2} + b_i(x) y^i\} / k,
\end{equation}
where,
\begin{align}\label{e18}
& a_{ij} = \frac{({\alpha}_{ik}x^k + {\beta}_i) ({\alpha}_{jk}x^k + {\beta}_j) - {\alpha}_{ij}({\alpha}_{km}x^kx^m + 2 {\beta}_k x^k +\gamma) }{({\alpha}_{km}x^kx^m + 2 {\beta}_k x^k + \gamma)^2},\\
\label{e19}
& b_j = - \frac{{\alpha}_{jk}x^k + {\beta}_j}{{\alpha}_{km}x^kx^m + 2 {\beta}_k x^k + \gamma} = -1/2\frac{\partial}{\partial x^j}(\log \phi (x)).
\end{align}
\ep
See \cite{11}, for a survey. If we restrict the Funk metric on the open interval $I = \{u\in\mathbb{R}\mid -1<u<1\}$, and let $\phi (x) = 1-x^2 $, then $D=\{x \in \mathbb{R}\mid \phi (x)>0 \}$, $\partial D = \{-1 , 1\}$. Similar argument determines the Funk metric on $I$ by
 \begin{equation}\label{e20}
L_f=\frac{1}{k} (\frac{\mid y \mid}{1-u^2}+\frac{uy}{1-u^2}).
 \end{equation}
 According to (\ref{distance}) the Funk distance of any two point $a$ and $b$ in $ I$ is given by
\begin{eqnarray}
D_f(a , b)
&&= \frac{1}{k} \mid \int_a^b \frac{du}{1-u^2} \mid + \frac{1}{k} \int_a^b \frac{u du}{1-u^2}\nonumber \\
&&= \frac{1}{2k} \mid [-\ln (1-u) ]_a^b + [\ln (1+u)]_a^b \mid - \frac{1}{2k} [\ln (1-u^2) ]_a^b \nonumber\\
&&\label{e21}
 = \frac{1}{2k} (\mid \ln \frac{(1-a)(1+b)}{(1-b)(1+a)} \mid + \ln \frac{(1-a^2)}{(1-b^2)}).
\end{eqnarray}
\subsection{Intrinsic pseudo-distance }
 A geodesic $f :I \rightarrow M$ on the Finsler space $(M,F)$ is said to be \emph{projective}, if the natural parameter $u$ on $I$ is a projective parameter.
We now come to the main step for definition of the pseudo-distance $d_M$, on $(M, F)$. To do so, we proceed in analogy with the treatment of Kobayashi in Riemannian geometry, cf., \cite{8}. Although he has confirmed that the construction of intrinsic pseudo-distance is valid for any manifold with an affine connection, or more generally a projective connection, cf., \cite{12},  we restrict our consideration to  the pseudo-distances induced by the Finsler structure $F$ on a  connected manifold $M$.
Given any two points $x$ and $y$ in $(M,F)$, we consider a chain $\alpha$ of geodesic segments joining these points. That is
\begin{itemize}
\item a chain of points $x = x_0 , x_1 , ... ,x_k = y$ on $M$;
\item pairs of points $a_1,b_1 ,..., a_k,b_k$ in $I$;
\item projective maps $f_1,...,f_k$, $f_i: I \rightarrow M $, such that
$$f_i(a_i) = x_{i-1}, \quad f_i(b_i) = x_i, \quad i = 1,...,k.$$
\end{itemize}
By virtue of the Funk distance $D_f(.,.)$ on $I$ we define the length $L(\alpha)$ of the chain $\alpha$ by

$L(\alpha):= \Sigma_i D_f(a_i , b_i)$, and we put
\begin{equation}\label{e19}
d_M(x , y):= inf L(\alpha),
\end{equation}
where the infimum is taken over all chains $\alpha$ of geodesic segments from $x$ to $y$.
\setcounter{thm}{0}
\bl
Let $(M, F)$ be a Finsler space. Then for any points $x$, $y$, and $z$ in $M$,  $d_M$ satisfies
 \begin{itemize}
 \item[(i)]\  $d_M(x , y) \neq d_M(y , x)$,
\item [(ii)]$d_M(x , z) \leq d_M(x , y) + d_M(y , z)$,
\item [(iii)]If $x = y$ then $d_M(x , y) = 0$ but the inverse is not  always true.
 \end{itemize}
\el
\begin{proof}
(i)  The Funk distance, $D_f(a_i , b_i) \neq D_f(b_i , a_i)$ for $i=1,...,k$. Therefore  we  have the proof of (i). To prove (ii),  it is enough to  show that for all positive $ \epsilon > 0$, the inequality  $ d_M(x , z) \leq d_M(x , y) + d_M(y , z) + \epsilon$ holds.
 There is a chain $\alpha_1$ joining the points $x$ and $y$ through the projective maps $f_i$, for $i=1,...,k_1$ and a chain $\alpha_2$ joining $y$ and $z$ through the projective maps $g_j$, for $j=1,...,k_2$ such that
  $$d_M(x , y) \leq L(\alpha_1) \leq d_M(x , y) +\epsilon /2,$$
  $$d_M(y , z) \leq L(\alpha_2) \leq d_M(y , z) +\epsilon /2.$$
  We define the chain $\alpha$ joining $x$ and $z$ through the projective maps $h_k$, for $k=1,...,k_1 + k_2$ such that
  $$h_k = f_k, \ \ \ k=1,...,k_1,$$
   $$h_k = g_{k - k_1}, \ \ \ k=k_1+1,...,k_1 + k_2.$$
   From which we conclude
   $$d_M(x , z) \leq L(\alpha) \leq L(\alpha_1) + L(\alpha_2) \leq d_M(x , y) + d_M(y , z) + \epsilon .$$
To prove (iii), using  the fact  $D_f(x,x)=0$, whenever $x=y$, we have $d_M(x,y)=0$.
   Next we assume $x\ne y$, and let $M=\mathbb{R}^n $. $\mathbb{R}^n$ is flat and $Ric=0$ on $\mathbb{R}^n$. Therefor the projective parameter $u$ of any projective map, according to (\ref{e16}), is given by $$u=\frac{b}{t+a}+c, \qquad a,b,c\in \mathbb{R},$$
   where $t$ is an arbitrary parameter.We use the special solution $u(t)=c$. Again according to the fact that $D_f(c,c)=0$, we have $d_M(x,y)=0$. This completes the proof.
   \end{proof}
We call $d_M(x,y)$ the \emph{ pseudo-distance} of any two points $x$ and $y$ on $M$.
From the property (\ref{PropertyOfSchwarzian}) of Schwarzian derivative, and the fact that the projective parameter is invariant under fractional transformation, the pseudo-distance $d_M$ is projectively invariant.

 \bp
Let $(M,F)$ be a Finsler space.
\begin{itemize}

 \item [(a)] If the geodesic  $f : I \rightarrow M $ is projective, then
 \begin{equation}\label{e23}
D_f(a , b) \geq d_M (f(a) , f(b)), \ \ \ \ a,b \in I.
 \end{equation}
 \item [(b)] If $\delta_M$ is any pseudo-distance on $M$ with the property
  $$D_f(a , b) \geq \delta_M(f(a) , f(b)), \ \ \ \ a,b \in I,$$
   and for all projective maps $f: I \rightarrow M$, then
  \begin{equation}\label{e24}
 \delta_M(x , y)\leq d_M(x , y), \ \ \ \ \ \  x,y \in M.
  \end{equation}
  \end{itemize}
  \begin{proof}
(a) By definition $d_M$ is supposed to be the infimum of $L(\alpha)$ for all chain $\alpha$  and actually $f$ is one of them.\\
(b) Let $x,y \in M$ and consider an arbitrary  chain of projective segments $\alpha$, satisfying
  $x = x_0,...,x_k = y$,   $a_1,b_1,...,a_k,b_k \in I$, and projective maps $f_1,...,f_k$, $f_i:I \rightarrow M$, such that
  $$f_i(a_i) = x_{i-1} \ \ \ , \ \ \ f_i(b_i) = x_i.$$
   We have by assumption
  $$L(\alpha) = \Sigma D_f(a_i , b_i) \geq \Sigma \delta_M (f(a_i) , f(b_i)).$$
  So for an arbitrary chain $\alpha$, the triangle inequality property for the pseudo-distance $\delta_M$ leads to
  $$L(\alpha) \geq \delta_M (f(a_1) , f(b_k))=\delta_M(x,y).$$
   Therefor $\delta_M(x , y)$ is a lower bound for $L(\alpha)$ and $inf_{\alpha} \ \ L(\alpha) \geq \delta_M(x , y)$. Finally we have $D_f(x , y) \geq \delta_M(x , y)$. This completes the proof.
  \end{proof}
  \ep
  \subsection{Proof of the Schwarz' lemma on Finsler Spaces}
  Let $ds_I=\frac{1}{k} (\frac{\mid dx \mid}{1-x^2} + \frac{x dx}{1-x^2})$ be the first fundamental  form related to the Funk metric $L_F$ on the open interval $I$, and $ds^2_M=g_{ij}(x , dx)dx^i dx^j$ the first fundamental  form related to the Finsler metric $F$ on $M$. Denote by  $\tilde{f}$  the natural lift of a projective map ${f}$ to the tangent bundle $TM$.
   Now we prove the Theorem 1.

{ \it Proof of Theorem 1.}
    Let $f : I\rightarrow M$ be an arbitrary projective map.
By virtue of (\ref{arc length}), we have
    $$s(t) = \int_{t_0}^t \sqrt{g_{ij}(f , \frac{df}{dt}) \frac{df^i}{dt} \frac{df^j}{dt}} dt.$$
    This is equivalent to
    \begin{equation}\label{e28}
    ds = \sqrt{g_{ij}(f , \frac{df}{dt}) df^i df^j}.
    \end{equation}
     We denote the projective parameter and the arc-length parameter of $f$, by ``$u$"  and  ``$s$", respectively.
    We put
    \be\label{h}
    h = \frac{ds^2_M}{ds_I},
    \ee
    and find an upper bound for $h$ in the open interval $I$. This leads to
     $$h = \frac{\sqrt{g_{ij}(f , df)df^i df^j}}{\frac{1}{k} (\frac{\mid du \mid}{1-u^2} + \frac{u du}{1-u^2})}.$$
     By means of (\ref{e28}) $h$ reads $h = \frac{k ds}{du(\frac{u \underline{+} 1}{1 - u^2})}.$
     Thus
    $$\ln h = \ln k + \ln \frac{ds}{du} + \ln \frac{1 - u^2}{u \underline{+} 1},$$
    $$\frac{d \ln h}{du}= \frac{s^{''}}{s^{'}} - \frac{2u}{(1 - u^2)} - \frac{1}{u \underline{+} 1}.$$
    At the maximum point of $h$, $\frac{d \ln h}{du}$ vanishes, so
     \begin{equation}\label{e29}
   \frac{s^{''}}{s^{'}} = \frac{2u}{(1 - u^2)} + \frac{1}{u \underline{+} 1}.
       \end{equation}
       The second derivative yields
       $$\frac{d^2\ln h}{du^2} = \frac{s^{'''}s^{'} - (s^{''})^2}{(s^{'})^2} + \frac{-2(1 - u^2) - 4u^2}{(1 - u^2)^2} + \frac{1}{(u \underline{+} 1)^2}$$
       $$\frac{s^{'''}}{s^{'}} - (\frac{s^{''}}{s^{'}})^2 - 2\frac{1 + u^2}{(1 - u^2)^2 } + \frac{1}{(u \underline{+} 1)^2}$$
       $$= \{s , u\} + \frac{1}{2}(\frac{s^{''}}{s^{'}})^2 - 2\frac{1 + u^2}{(1- u^2)^2} + \frac{1}{(u \underline{+} 1)^2}.$$
       By virtue of  (\ref{e29}) and (\ref{schwarzian}), the parameters $p$ and $t$ satisfy
       $$\{p , t\} = - \{t , p\} (\frac{dp}{dt})^2.$$
       Thus we get
        $$\frac{d^2\ln h}{du^2} = - \{u , s\} (\frac{ds}{du})^2 - \frac{2}{(1 - u^2)^2} +  \frac{3}{2(u \underline{+} 1)^2} + \frac{2u}{(1 - u^2)(u \underline{+} 1)}. $$
         Also, at the maximum point of $h$, $\frac{d^2\ln h}{du^2} \leq 0$. Considering this fact, by multiplying  both side of the inequality in  $\frac{k^2(1 - u^2)^2}{(u \underline{+} 1)^2}$, we have
          $$- \{u , s\} h^2 -2 \frac{k^2}{(u \underline{+} 1)^2} + \frac{3}{2} \frac{k^2(1 - u^2)^2}{(u \underline{+} 1)^4} + 2 \frac{k^2 u(1 - u^2)}{(u \underline{+} 1)^3} \leq 0.$$
          To study the above statement we examine the two cases $u+1$ and $u-1$.
          \emph{Case (I):} Let us consider the term $(u + 1)$. Thus
          $$- \{u , s\} h^2 + k^2 (\frac{-2 + 3/2(1 - u)^2 + 2u(1 - u)}{(u + 1)^2}) \leq 0,$$
          which reduces to
          \begin{equation}\label{e30}
          - \{u , s\} h^2 \leq \frac{k^2}{2}.
          \end{equation}
          On the other hand, we have
          $$\{u , s\} = \frac{2}{n-1}(Ric)_{ij} \frac{dx^i}{ds} \frac{dx^j}{ds}.$$
           By means of the assumption (\ref{e25}), we get
           $$\{u , s\} \leq \frac{-2c^2}{n-1}g_{ij} \frac{dx^i}{ds} \frac{dx^j}{ds}.$$
           For the arc-lenght parameter $s$, we have $g_{ij} \frac{dx^i}{ds} \frac{dx^j}{ds} = 1$, Therefor
            \begin{equation}\label{e31}
          \{u , s\} \leq \frac{-2c^2}{n-1} \leq 0.
             \end{equation}
          Taking into account (\ref{e30}) and (\ref{e31}), $h$ satisfies the following inequality
          $$h^2 \leq \frac{-k^2}{2\{u , s\}}.$$
          Finally we have
          \be\label{h2}
          h^2 \leq \frac{k^2 (n - 1)}{4c^2}.
          \ee
          \emph{Case (II):} For  the term $(u - 1)$ by similar argument  we obtain (\ref{h2}). This completes the proof.
 \hspace{\stretch{1}}$\Box$\\
 \begin{cor}
          Let $(M,F)$ be a Finsler space for which the Ricci tensor satisfies
  \begin{equation}\label{e25}
 (Ric)_{ij} \leq -c^2g_{ij}.
  \end{equation}
  as matrices, for a positive constant $c$. Let $d_F(. , .)$ be the distance induced by $F$, then for every projective map $f:I \rightarrow M$, $d_F$ is bounded by the Funk distance $D_f$, that is
   \begin{equation}\label{e32}
D_f(a , b) \geq \frac{2 c}{\sqrt{n - 1} k} d_F(f(a) , f(b)).
 \end{equation}
 \end{cor}
\begin{proof}
By means of Theorem 1 we have
 $$(ds)^2 \leq \frac{k^2 (n - 1)}{4c^2}ds_I^2,$$
 that is $\frac{2 c}{\sqrt{n - 1} k} ds \leq ds_I.$
 By integration we obtain(\ref{e32}).
\end{proof}

Now we are in a position to prove Theorem 2.

 \textit{Proof of Theorem 2.}
To establish the proof we have only to show that if  $d_M(x , y) = 0$ then $x = y$.
  By Proposition 2 and the above corollary we get
  $$ d_F(x , y) \frac{2 c}{\sqrt{n - 1} k} \leq d_M(x , y).$$
If  $d_M(x , y) = 0$ then $d_F(x , y) = 0$ and $x = y$. Thus the pseudo-distance $d_M$ is a distance. This completes the proof.
 \hspace{\stretch{1}}$\Box$\\

  \end{document}